\documentclass[11pt,oneside]{amsart}
\usepackage[utf8]{inputenc}
\usepackage[greek,english]{babel}
\usepackage{alphabeta}
\usepackage{amsmath}
\usepackage{amssymb}
\usepackage{amsthm}
\usepackage{physics}
\usepackage{hyperref}
\hypersetup{
	colorlinks=true,
	linkcolor=blue,
	filecolor=magenta,
	urlcolor=cyan,
}
\newtheorem{theorem}{Theorem}[section]
\newtheorem{definition}[theorem]{Definition}
\newtheorem{proposition}[theorem]{Proposition}
\newtheorem{lemma}[theorem]{Lemma}
\theoremstyle{definition}
\newtheorem{defn}[theorem]{Definition}

\newtheorem{example}[theorem]{Example}

\def\Aa{\mathcal {AA}}
\def\Hor{\mathcal{H}}
\def\divH{\mathrm{div}_{H}}
\def\dive{\mathrm{div}_{\epsilon}}

\def\nablae{\nabla^{\epsilon}}

\def\R{\mathbb{R}}
\def\C{\mathbb{C}}

\numberwithin{equation}{section}

\begin{document}
	\title[Isoperimetric inequality on 3D-contact, non-unimodular Lie groups]{An isoperimetric inequality on three-dimensional, contact,  non-unimodular Lie groups}

	\author{ Ioannis D. Platis \and Georgios I. Simantiras}
	\begin{abstract}
	We prove an isoperimetric inequality for compact bodies bounded by surfaces embedded into a connected, contact, non-unimodular 3-dimensional Lie group.
	\end{abstract}
	\maketitle
	
	\newcommand{\Addresses}{%
		\bigskip\bigskip
		\footnotesize
		
		\noindent
		I.D.~Platis, \textsc{Department of Mathematics, University of Patras, Panepistimioupolis, 26504 Rion, Achaia, Greece}\par
		\textit{E-mail address}: \href{mailto:idplatis@upatras.gr}{idplatis@upatras.gr}
		
		\medskip
		
		\noindent
		G.I.~Simantiras, \textsc{Department of Mathematics, University of Patras, Panepistimioupolis, 26504 Rion, Achaia, Greece}\par
		\textit{E-mail address}: \href{mailto:gsimantiras@upatras.gr}{gsimantiras@upatras.gr}
		
		\bigskip
	}
	
	\date{\today}

\section{Introduction}
The study of sub-Riemannian geometry on Lie groups provides profound insights into spaces equipped with completely non-integrable distributions. In this paper, we investigate the geometric and measure-theoretic properties of three-dimensional, connected Lie groups $G$ equipped with a contact form $\theta$. A central feature of such spaces is the horizontal distribution $\mathcal{H} = \ker\theta$, which is completely non-integrable and generates the entire tangent bundle through its Lie brackets. By endowing $\mathcal{H}$ with a left-invariant sub-Riemannian metric $g_{CC}$, the group $G$ becomes a metric space governed by the Carnot-Carathéodory distance $d_{CC}$. 

A critical algebraic classification of these groups relies on their unimodularity. A Lie group $G$ is unimodular if the trace of the adjoint representation vanishes for every vector field in its Lie algebra. In this work, we focus specifically on the non-unimodular case, where this trace is non-zero for at least one left-invariant vector field. Two characteristic examples illustrating this dichotomy are the Heisenberg group, which is unimodular, and the Affine-Additive group, which is non-unimodular. 

To study the volume and perimeter of domains within $G$, we employ a technique of Riemannian approximation. By introducing a family of Riemannian metrics $g_\epsilon$ that rescale the Reeb field $T$ by a parameter $\epsilon > 0$, we create a sequence of metric spaces $(G, d_\epsilon)$ that converge to the sub-Riemannian space $(G, d_{CC})$ in the pointed Gromov-Hausdorff sense as $\epsilon \to 0^+$. This approximation allows us to adapt classical Riemannian tools, such as the Levi-Civita connection and Gauss' Divergence Theorem, to the sub-Riemannian setting by analyzing the limit as $\epsilon \to 0^+$. 

The main result of this paper is an isoperimetric inequality for three-dimensional, contact, non-unimodular Lie groups. We prove in Theorem \eqref{ISOproof} that for every $\mu_G$-measurable domain $\Omega \subset G$ with $\mu_G$-volume $\mathrm{Vol}_{G}(\Omega)$ bounded by a $C^1$-surface $\Sigma = \partial\Omega$ with perimeter ${\rm Per}_G(\Sigma)$ (see Definition \eqref{def-perimeter}), the following inequality holds:
$$ \mathrm{Vol}_{G}(\Omega) \leq C \cdot \mathrm{Per}_{G}(\partial\Sigma), $$
where $C = C(G)$ is a positive constant.  Furthermore, we examine the strictness of this bound and show that for any domain of finite volume, equality cannot be achieved due to geometric obstructions arising from the rigidity of left-invariant horizontal normals.

This inequality may be viewed as a sub-Riemannian version of the seminal result of Gromov, i.e.,  the linear isoperimetric inequality $$\text{Area}(\Omega) \le C \cdot \text{Length}(\partial\Omega).$$ Gromov established this inequality as the defining geometric property of the hyperbolic plane and its coarse generalisations in his work on hyperbolic groups, \cite{Gromov87}, see also \cite{Gromov83}.

\section{Preliminaries}
\subsection{Three dimensional contact sub-Riemannian Lie groups}
Let $G$ be an arbitrary 3-dimensional connected Lie group $G$; we denote by $e$ the identity element of $G$ and also, for $g\in G$ we denote by $L_g:G\to G$ the left translation given by $L_g(h)=gh,\, h\in G$. We shall assume that $G$ is equipped with a contact form $\theta$, i.e., a $1$-form such that $\theta\wedge d\theta\neq 0$. Due to continuity we may always assume that $\theta\wedge d\theta>0$ and then 
$d\mu_{G}:=\theta\wedge d\theta$ is the left-invariant Haar measure on $G$. The horizontal distribution $\Hor$ in the tangent bundle ${\mathrm T}\,G$ is given by $\Hor:=\ker\theta={\rm span}\{X,Y\}$, where $X$ and $Y$ are left-invariant vector fields. Since $\theta$ is a contact form, the distribution $\Hor$ is completely non-integrable: ${\mathrm T}\,G=\Hor+[\Hor,\Hor]$. We denote by $T$ the Reeb field, i.e., the left-invariant vector field such that $\theta(T)=1$ and $i_{T}d\theta\,(\cdot)=0$. A left-invariant frame for ${\mathrm T}\,G$ is given by $\{X,Y,T\}$ and we assign a coframe $\{\omega_1,\omega_2,\theta\}$ given by the dual 1-forms to $\{X,Y,T\}$. We endow $G$ with a CR structure $J:\Hor\to \Hor$ defined by $JX=Y$, $JY=-X$ which we extend to the whole tangent bundle by setting $JT=0$. We denote by $\mathfrak{g}$, the Lie algebra of $G$:
$$\mathfrak{g}={\rm T}(G)=\mathrm{span}\{X,Y,T\}=\Hor \oplus \mathrm{span}\{T\}.$$
The {\it sub-Riemannian metric} $g_{CC}( \cdot\, , \cdot ) :\mathcal{H}\times\mathcal{H}\to \R$ is the left-invariant metric in $G$ such that $\{X,Y\}$ is an orthonormal frame:
$$g_{CC} (X,X)=g_{CC}(Y,Y)=1 \quad\text{and} \quad g_{CC}( X,Y)=0.$$
An absolutely continuous curve $\gamma:[a,b]\to G$ shall be called {\it horizontal} if $\dot{\gamma}(s)\in\Hor_{\gamma(s)}$ for almost every $s\in[a,b]$. Then, the {\it horizontal velocity} $\norm{\dot\gamma(s)}_{\Hor}$ of $\gamma$ is defined by
$$\norm{\dot\gamma(s)}_{\Hor}=\sqrt{g_{CC}(\dot\gamma(s),X_{\gamma(s)})^2+g_{CC}(\dot\gamma(s),Y_{\gamma(s)})^2}\,,$$
and the {\it horizontal length} \(\ell_H(\gamma)\) of \(\gamma\) is:
 $$\ell_H(\gamma)=\int_a^b\norm{\dot\gamma(s)}_{\Hor}\,ds.$$
For every $p, q \in G$, the {\it Carnot-Carath\'eodory or sub-Riemannian distance} $d_{CC}$ associated with the sub-Riemannian metric $g_{CC}$ is defined by: 
\begin{equation}\label{CC distance}
    d_{CC}(p,q):=\inf_{\gamma\in \Gamma_{(p,q)}}\{\ell_H(\gamma) \},
\end{equation}
where $\Gamma_{(p,q)}=\{\gamma,\gamma:[0,1]\to G \text{ horizontal and } \gamma(0)=p,\, \gamma(1)=q\}$. Definition \eqref{CC distance} depends only on the values of $g_{CC}$ in $\mathcal{H}\times \mathcal{H}$. Moreover, since $\mathcal{H}$ is completely non-integrable, the distance $d_{CC}$ is finite, geodesic, and induces the manifold topology (cf. \cite{Mont}).

\subsection{Unimodular and non-unimodular groups}
Let $\mathrm{ad} : \mathfrak{g}\longrightarrow \mathrm{End}(\mathfrak{g})$, $Z\mapsto \mathrm{ad}_Z:=[Z\,,\,\cdot]$, be the \emph{adjoint representation } of the Lie algebra $\mathfrak{g}$ of $G$ and let $\mathrm{Ad} : G\longrightarrow \mathrm{Aut}(\mathfrak{g})$, $g\mapsto \mathrm{Ad}_g:=D(c_g)_{e}$ be the \emph{adjoint representation} of $G$. Recall that a connected Lie group $G$ is {\it unimodular} if ${\rm tr}({\rm ad}_Z)=0$ for each $Z\in\mathfrak{g}$; equivalently, if $\det({\rm Ad}_g)=\pm1$, for each $g\in G$.

For all $g\in G$, the diffeomorphism $c_{g}(x)=gxg^{-1}$ is the \emph{conjugation map by} $g$, while $g^{-1}$ is its inverse element on $G$. In our particular case, the left-invariant vector fields $X,Y,T$ defined above satisfy the following (see \cite{Platis2026}, Proposition 2.1): there exist constants $a_i,b_i,c_i\in \R$, $i=1,2,3$ and a constant $c<0$, such that:
\begin{equation}\label{3brackets}
    [X,T]=a_1X+b_1Y, \quad [Y,T]=a_2X+b_2Y, \quad [X,Y]=a_3X+b_3Y+cT.
\end{equation}
Also,
\begin{equation}\label{eqsab}
 \left|\begin{matrix}
    a_1&b_1\\
    a_3&b_3
\end{matrix}\right|=
\left|\begin{matrix}
    a_2&b_2\\
    a_3&b_3
\end{matrix}\right|=0,\quad a_1+b_2=0.\end{equation}
From Eqs. \eqref{3brackets} and \eqref{eqsab} we straightforwardly obtain:
\begin{equation}\label{ads}
    \mathrm{ad}_X=\begin{bmatrix}
        0 & a_3 & a_1\\
        0 & b_3 & b_1\\
        0 & c & 0
    \end{bmatrix}, \, \mathrm{ad}_Y=\begin{bmatrix}
        -a_3 & 0 & a_2\\
        -b_3 & 0 & b_2\\
        -c & 0 & 0
    \end{bmatrix}, \, 
     \mathrm{ad}_T=\begin{bmatrix}
        -a_1 & -a_2& 0\\
        -b_1 & -b_2 & 0\\
        0 & 0 & 0
    \end{bmatrix},
\end{equation}
and by the $\R$-linearity of the trace we deduce
$$
\mathrm{tr}\left(\mathrm{ad}_{Z_0}\right)= fb_3-ga_3-h(a_1+b_2)=fb_3-ga_3,
$$
for any left-invariant vector field $Z_0=fX+gY+hT$, $f,g,h\in\R$.

\begin{proposition}\label{prop-unim}
Let $G$ be as above. Then $G$ is unimodular if and only if $a_3=b_3=0$. Equivalently $G$ is not unimodular if and only if $|a_3|+|b_3|\neq 0$.
\end{proposition}
\begin{proof}
If $G$ is unimodular then for each $Z_0$ as above we have $fb_3=ga_3$. Since $f,g$ are arbitrary, we obtain $a_3=b_3=0$. The converse is obvious.
\end{proof}
We give below two characteristic examples for each case.

\subsubsection{The Heisenberg group}
The Heisenberg group $\mathbb{H}$ is $\R^3$ with coordinates $(x,y,t)$ and group multiplication
$$
p\cdot p'=(x+x',\, y+y',\, t+t'+2(xy'-x'y)),
$$
for each $p=(x,y,t)$ and $p'=(x',y',t')$. The contact form is
$$
\theta=dt+2x\,dy-2y\,dx,
$$
the volume form is
$$
d\mu_{\mathbb{H}}=\theta\wedge d\theta=4\,dx\wedge dy\wedge dt,
$$
and the left invariant vector fields are
$$
X=\partial_x+2y\partial_t,\quad Y=\partial_y-2x\partial_t,\quad T=\partial_t.
$$
The only non zero Lie bracket is $[X,Y]=-4T$, therefore $\mathbb{H}$ is unimodular.

\subsubsection{The Affine-Additive group}\label{AAexample1}
For details, see \cite{BThesis}. The Affine-Additive group $\Aa$ is $\R \times \mathbf{H}^1_{\mathbb{C}},$ with coordinates $(a,\lambda,t)$, where $\mathbf{H}^1_\C=\{(\lambda,t) \, | \,\lambda>0,\,t\in\R\},$ is the right half-plane model for the hyperbolic plane. The group multiplication is defined by
$$ p \cdot p'=(a+a',\,\lambda \lambda' ,\, \lambda t'+t),
$$
for every $p=(a,\lambda,t)$ and $p'=(a',\lambda',t')\in \Aa$. The contact form is
$$
\theta=da-\frac{1}{\lambda}dt, 
$$
the volume form is
$$
d\mu_{\Aa}=\frac{1}{\lambda^2}da\,\wedge\, d\lambda\,\wedge dt,
$$
and the left-invariant frame comprises
\begin{equation*}
     U=\partial_{a}+\lambda\partial_{t}, \, V=\lambda\partial_{\lambda}, \quad W=\partial_{a}.
\end{equation*}
The Lie brackets are:
\begin{equation*}
    [U,V]=-U+W, \quad [U,W]=[V,W]=0,
\end{equation*}
hence in this case we have $a_3=-1$ and $\Aa$ is non-unimodular.

\subsection{Horizontal divergence}
\subsubsection{Riemannian approximation}
For $\epsilon>0$, we consider a family of Riemannian metrics $g_\epsilon$ on $G$ such that the frame $\{X,Y,T_\epsilon=\epsilon T\}$ is orthonormal and we denote its corresponding coframe by $\{\omega_1,\omega_2,\theta_\epsilon=(1/\epsilon)\theta\}$. Let $d_\epsilon$ be the associated Riemannian distance to $g_\epsilon$. The following theorem holds \cite{CDPT, Gromov}:
\begin{theorem}
The family of metric spaces $(G,d_\epsilon)$ converges to the metric space $(G,d_{CC})$ in the pointed Gromov-Hausdorff sense as $\epsilon\to 0^+$.
\end{theorem}
The first structural equations for the Levi-Civita connection $\nabla^\epsilon$ of $g_\epsilon$ are
\begin{eqnarray*}
d\omega_1&=&\theta_1^2\wedge\omega_2+\theta_1^3\wedge\theta_\epsilon,\\
d\omega_2&=&-\theta_1^2\wedge\omega_1+\theta_2^3\wedge\theta_\epsilon,\\
d\theta_\epsilon&=&-\theta_1^3\wedge\omega_1-\theta_2^3\wedge\omega_2,
\end{eqnarray*}
where
\begin{eqnarray*}
   \theta_1^2&=&-a_3\,\omega_1-b_3\,\omega_2+\left(\frac{\epsilon}{2}(a_2-b_1)-\frac{c}{2\epsilon}\right)\,\theta_\epsilon,\\
   \theta_1^3&=&-\epsilon a_1\,\omega_1-\left(\frac{\epsilon}{2}(b_1+a_2)+\frac{c}{2\epsilon}\right)\,\omega_2,\\
   \theta_2^3&=&\left(\frac{c}{2\epsilon}-\frac{\epsilon}{2}(b_1+a_2)\right)\,\omega_1-\epsilon b_2\,\omega_2,
\end{eqnarray*}
see \cite{Platis2026}. For an arbitrary vector field $U$, the covariant derivatives of $X_1=X$, $X_2=Y$ and $X_3=T_\epsilon$ with respect to $U$ are given by
$$
\nabla^\epsilon_UX_j=\sum_{i=1}^3\theta^i_{j}(U)X_i, \quad j=1,2,3.
$$
Straightforward calculations deduce the following:
\begin{eqnarray*}
\nabla^\epsilon_XX&=&-a_3\,Y-\epsilon a_1\,T_\epsilon,\\
 \nablae_{X}Y &=&a_3\,X+\left(\frac{c}{2\epsilon}-\frac{\epsilon}{2}(b_1+a_2)\right)\,T_{\epsilon},\\ 
 \nabla^\epsilon_XT_\epsilon&=&\epsilon a_1\,X+\left(-\frac{c}{2\epsilon}+\frac{\epsilon}{2}(a_2+b_1)\right)\,Y,\\
 \nablae_{Y}X&=&-b_3\,Y+\left(-\frac{c}{2\epsilon}-\frac{\epsilon}{2}(b_1+a_2)\right)\,T_{\epsilon},\\
 \nablae_{Y}Y &=&b_3\,X-\epsilon b_2\, T_\epsilon,\\
 \nablae_{T_{\epsilon}}X&=&\left(\frac{\epsilon}{2}(a_2-b_1)-\frac{c}{2\epsilon}\right)\,Y,\\
 \nablae_{T_\epsilon}Y&=&\left(\frac{c}{2\epsilon}-\frac{\epsilon}{2}(a_2-b_1)\right)\,X,\\
 \nablae_{T_\epsilon}T_\epsilon&=&0.
\end{eqnarray*}

\subsubsection{Horizontal divergence of vector fields}
If $Z$ is a vector field in $G$, then 
\begin{equation}
Z=fX+gY+hT_{\epsilon}, \quad \text{where} \quad f,g,h\in C^{\infty}(G),\label{Zvf} 
\end{equation}
and the $g_\epsilon$-divergence ${\rm div}_\epsilon(Z)$ of $Z$ is given by (see, e.g., \cite{GallotHulinLafontaine})
\begin{equation}
    \mathrm{div}_{\epsilon}(Z)=g_\epsilon(\nablae_{X}Z,X) + g_\epsilon(\nablae_{Y}Z,Y)+g_\epsilon( \nablae_{T_{\epsilon}}Z,T_\epsilon).\label{RiemDiv}
\end{equation}
Evaluating the inner products yields:
\begin{equation}
    g_\epsilon(\nablae_{X}Z,X)=X(f)+ga_3+\epsilon h a_1,\label{inner1}
\end{equation}
\begin{equation}
   g_\epsilon(\nablae_{Y}Z,Y)=Y(g)+h \epsilon b_2-fb_3,
\end{equation}
\begin{equation}
    g_\epsilon(\nablae_{T_\epsilon}Z,T_\epsilon)=T_\epsilon(h).\label{inner3}
\end{equation}
Substituting (\ref{inner1})-(\ref{inner3}) back into (\ref{RiemDiv}) and using $a_1+b_2=0$, we conclude:
\begin{equation}
   \mathrm{div}_{\epsilon}(Z)=X(f)+Y(g)+\epsilon T(h)+ga_3-fb_3.\label{divRZ2} 
\end{equation}
By letting $\epsilon\to 0^+$ we have the following:
\begin{definition}
    The horizontal divergence of a vector field $Z$ as in \eqref{Zvf} is defined by the formula:
    \begin{equation}
        \mathrm{div}_{\Hor}^{G}(Z)=X(f)+Y(g)+ga_3-fb_3. \label{divh}
    \end{equation}
\end{definition}
In particular, when $Z$ is left-invariant, the functions $f,g,h$ are real constants; in this case,
$$
\dive(Z)=\divH^G(Z)=ga_3-fb_3, \quad Z\in \mathfrak{g}.
$$

\subsection{Hypersurfaces, volume and horizontal area}
Let $u:G\to\R$ be a $C^2$-smooth function such that its tangent map is surjective for all $p\in u^{-1}(0)$. Then, the regular level set theorem \cite{Lee} implies that the set 
\begin{equation*}
    \Sigma=\{p\in G:u(p)=0\},
\end{equation*}
is a regular (embedded) 2-dimensional submanifold of $G$, i.e., a \emph{hypersurface} in $G$. We say that a point $p \in\Sigma$ is \textit{characteristic} if
\begin{equation}
    \nabla_H u(p):=\left(X(u)X+Y(u)Y\right)\bigr\rvert_p=0,
\end{equation}
i.e., if its \emph{horizontal gradient} vanishes at $p$. The \textit{characteristic set} $\mathcal{C}(\Sigma)$ of $\Sigma$ comprises all characteristic points. We set $p=X(u)$, $q=Y(u)$, $r=T_\epsilon (u)$, and introduce the notation:
\begin{equation}
\begin{aligned}\label{notation MF 0}
    l=\|\nabla_{\Hor} u\|:=\sqrt{\left(X(u)\right)^2+\left(Y(u)\right)^2},\quad \overline{p}=\frac{p}{l},\quad\overline{q}=\frac{q}{l},\quad \overline{r}=\frac{r}{l}\\
    l_\epsilon:=\sqrt{\left(X(u)\right)^2+\left(Y(u)\right)^2+\left(T_\epsilon (u)\right)^2},\quad \overline{r_\epsilon}=\frac{r}{l_\epsilon},\\
    \overline{p_\epsilon}=\frac{p}{l_\epsilon}\,\quad\text{and}\quad \overline{q_\epsilon}=\frac{q}{l_\epsilon}.
\end{aligned}
\end{equation}
Note that the functions $\bar{p}, \bar{q}$ and $\bar{r}$ are well defined away from the set $\mathcal{C}(\Sigma)$. For further reference, we will need the following.
\begin{lemma}\label{l,r limits}
    Let $\overline{p}$, $\overline{q}$, $l_\epsilon$ and $\overline{r_\epsilon}$ be as above. Then, when $\epsilon\to 0^+$, the following limits hold:
     \begin{eqnarray*}
        l_\epsilon\to\|\nabla_H u\|,\quad \overline{r_\epsilon}\to 0,\quad \frac{\overline{r_\epsilon}}{l_\epsilon}\to 0,\\
        \frac{\overline{r_\epsilon}}{\epsilon\, l_\epsilon}\to\frac{Tu}{\|\nabla_H u\|^2},\quad \left(\frac{\overline{r_\epsilon}}{\epsilon}\right)^2\to\frac{(Tu)^2}{\|\nabla_H u\|^2},\quad \frac{\overline{r_\epsilon}}{\epsilon^2}\sim\frac{Tu}{\epsilon \|\nabla_H u\|}.
\end{eqnarray*}
\end{lemma}
\begin{defn}
The {\it horizontal unit normal} associated to the surface $\Sigma$ is the vector field
\begin{equation*}
    \nu_H=\frac{\nabla_H u}{\norm{\nabla_H u}}=\overline{p}X+\overline{q}Y.
\end{equation*}
\end{defn}

\begin{defn}\label{E_1, E_2}
    Let $N_{\epsilon} = \overline{p_\epsilon}\,X+\overline{q_\epsilon}\,Y+\overline{r_\epsilon}\,T_\epsilon$ be the Riemannian unit normal to the surface $\Sigma$ and let also
\begin{equation*} 
    E_1=\frac{l}{l_\epsilon}(\overline{r}\,\overline{p}\,X+\overline{r}\,\overline{q}\,Y-T_\epsilon),\quad E_2=J\nu_H=-\overline{q}\,X+\overline{p}\,Y,
\end{equation*}
where $J:\mathrm{T}\,G\to\mathrm{T}\,G$ is the CR structure of $G$. Then $\{E_1,E_2,N_{\epsilon}\}$ is an orthonormal frame w.r.t $g_\epsilon$ associated to $\Sigma$.
\end{defn}

\begin{definition}\label{E_1, E_2, dual}
 The coframe dual to the frame of Definition \ref{E_1, E_2} comprises the differential $1$-forms $\{\alpha_1,\alpha_2,\alpha_3\}$, given by
\begin{equation*}
    \begin{aligned}
    \alpha_1&=(l/l_\epsilon)(\overline{r}\,\overline{p}\,\omega_1+\overline{r}\,\overline{q}\,\omega_2-\theta_\epsilon),\\
    \alpha_2&= -\overline{q}\,\omega_1+\overline{p}\,\omega_2,\\
    \alpha_3&=(l/l_\epsilon)\left(\overline{p}\,\omega_1+\overline{q}\,\omega_2+\overline{r}\,\theta_\epsilon\right).
\end{aligned}
\end{equation*}
\end{definition}
The Riemannian volume form for $(G,g_\epsilon)$ is given by $dV_\epsilon=\alpha_1\wedge \alpha_2\wedge \alpha_3=\omega_1\wedge \omega_2\wedge\theta_{\epsilon}$. For the Haar measure, we have:
\begin{equation}
    d\mu_{G}=(-c)\,\epsilon dV_{\epsilon}>0. \label{HaarVol}
\end{equation}
On the hypersurface $\Sigma$, the induced area element $d\Sigma_{\epsilon}:=i_{N_{\epsilon}}dV_{\epsilon}\bigr\rvert_{\Sigma}$ simplifies to $\alpha_1\wedge\alpha_2$.

\begin{lemma}\label{lem-i}
If $\nu_H$ is the unit horizontal vector field to the surface $\Sigma$, then
\begin{equation}
i_{\nu_H}d\mu_G=(-c)\lim_{\epsilon\to 0^+}\epsilon\,d\Sigma_\epsilon.
\end{equation}
\end{lemma}
\begin{proof}
Direct contraction yields:
\begin{eqnarray*}
i_{\nu_H}d\mu_G&=&(-c)i_{\overline{p}\,X+\overline{q}\,Y}(\omega_1\wedge\omega_2\wedge\theta) =(-c)\left(\overline{p}\,\omega_2\wedge\theta-\overline{q}\,\omega_1\wedge\theta\right) =(-c)\lim_{\epsilon\to 0^+}\epsilon\,d\Sigma_\epsilon.
\end{eqnarray*}
\end{proof}
We assume that $\Sigma$ bounds a measurable domain $\Omega\subset G$, i.e., $\Sigma=\partial\Omega$. 
\begin{definition}\cite{Shcherbakova2009}\label{def-perimeter}
The {\it horizontal area} or {\it perimeter} of a $C^1$-surface $\Sigma = \partial \Omega$ is defined by
\begin{equation}
    \mathrm{Per}_{G}(\Sigma):=\int_{\Sigma}i_{\nu_H}\,d\mu_G\bigr\rvert_{\Sigma} = (-c)\lim_{\epsilon\to 0^+}\int_{\Sigma}\epsilon\,d\Sigma_\epsilon.\label{per1}
\end{equation}
\end{definition}
In what follows, we shall write
$$
d{\rm Per}_G(\Sigma):=(-c)\lim_{\epsilon\to 0^+}d\Sigma_\epsilon.
$$

\section{An isoperimetric inequality in the non-unimodular case}\label{PROOFISO}
In this section, we prove Theorem \ref{ISOproof}, adapting the Riemannian divergence methods found in \cite{Faessler2025, Pittet2000}.

\begin{theorem}\label{ISOproof}
Let $G$ be a three-dimensional, connected contact Lie group which is non-compact and non-unimodular. Suppose that $\mu_G$ is the Haar measure in $G$. Then for every $\mu_G$-measurable domain $\Omega\subset G$ bounded by a $C^1$-surface $\Sigma=\partial\Omega$, the following isoperimetric inequality holds:
\begin{equation}
    \mathrm{Vol}_{G}(\Omega)\le C\cdot\mathrm{Per}_{G}\left(\Sigma\right),\label{ISO}
\end{equation}
where $C=C(G)$ is a positive constant.
\end{theorem} 
\begin{proof}
Let $Z\in \mathfrak{g}$ be a left-invariant vector field. Gauss' divergence theorem on $(\Omega, g_\epsilon)$ yields
\begin{equation}
    \int_{\Omega}\dive(Z)\, \epsilon \, dV_\epsilon=\int_{\Sigma}g_\epsilon (Z, N_\epsilon)\, \epsilon\, d\Sigma_\epsilon.\label{Gauss2}
\end{equation}
Now,
					\begin{equation}
						\lim_{\epsilon\to 0^+}\left(\dive(Z)\, \epsilon \, dV_\epsilon\right)=\frac{\divH(Z)}{(-c)}\, d\mu_G, \label{Int1}
					\end{equation}
					whereas 
					\begin{equation}
						\lim_{\epsilon\to 0^+}g_\epsilon( Z, N_\epsilon)\, \epsilon\, d\Sigma_\epsilon=\frac{g_{CC}( Z, \nu_H)}{(-c)}\, d{\rm Per}_G.\label{int2}
					\end{equation}
Taking the limits as $\epsilon \to 0^+$ in \eqref{Gauss2} and using equations \eqref{Int1} and \eqref{int2}, we obtain:
\begin{equation*}
    \int_{\Omega}\divH^G(Z)\,d\mu_G=\int_{\Sigma}g_{CC}( Z_H,\nu_H) \, d{\rm Per}_G \leq \int_{\partial\Omega}\norm{Z_H}\, d{\rm Per}_G,
\end{equation*}
where the last line follows from the Cauchy-Schwarz inequality since $\norm{\nu_H}=1$.

Since $G$ is non-unimodular, there exists a left-invariant vector field $Z_0 \in \mathfrak{g}$ such that 
$$
\divH^G(Z_0) = ga_3-fb_3 = -\mathrm{tr}(\mathrm{ad}_{Z_0}) \neq 0.
$$ Normalising its horizontal component so that $g=\cos\phi$ and $f=\sin\phi$, we get:
$$ |a_3\cos\phi-b_3\sin\phi|\mathrm{Vol}_G(\Omega) \le \mathrm{Per}_G(\partial\Omega). $$
Taking the supremum over all $\phi \in \R$ yields $\sup_{\phi} |a_3\cos\phi-b_3\sin\phi| = \sqrt{a_3^2+b_3^2} := 1/C$, we have
$$
\mathrm{Vol}_{G}(\Omega)\le C\cdot\mathrm{Per}_{G}\left(\Sigma\right).
$$
\end{proof}
We finally prove that this inequality is strict for any non-trivial domain $\Omega$ with finite volume ($\mathrm{Vol}_G(\Omega) < \infty$).
In the first place, the surfaces $\Sigma$ for which equality may occur in the Cauchy-Schwarz step must satisfy $\nu_H = \pm Z_{H}$ away from characteristic points. If equality holds, this forces $\nu_H$ to be equal to a constant left-invariant vector field almost everywhere. This constrains the boundary $\Sigma= \{u=0\}$ to satisfy a linear rigid PDE of the form $Z_H(u)=0$. The integral surfaces of such equations are invariant under non-compact flows extending to infinity, meaning they cannot enclose a domain $\Omega$ of finite non-zero volume.
\begin{example}
In the Affine-Additive group $\Aa$ (see Section \ref{AAexample1}), $a_3=-1$ and $b_3=0$, which gives $C_{\mathrm{best}}=1$. The isoperimetric inequality states:
\begin{equation}
    \mathrm{Vol}_{\Aa}(\Omega) < \mathrm{Per}_{\Aa}(\partial\Omega), \label{aaiso}
\end{equation}
where the inequality is strict for all bounded $\Omega$. Any hypothetical minimiser would have to satisfy the structural boundary equation away from its characteristic locus:
\begin{equation}
U(u)=0 \iff u_a+\lambda\,u_t=0. \label{eqODE}
\end{equation}
The global solutions to this are cylinders of the form $u(a,\lambda,t)=\Phi(\lambda, t-\lambda\cdot a)=0$, which are non-compact and fail to bound finite-volume domains.
\end{example}

\bibliographystyle{abbrv}
				\bibliography{references}
				\bigskip
				\Addresses
\end{document}